\documentclass[]{siamltex}
	\usepackage{epsfig}
	\usepackage{psfrag}
	\usepackage{a4wide}
	\usepackage{amsmath}
	\usepackage{amssymb}
	\usepackage{subfigure}
	\usepackage{comment}
	\usepackage[bottom=2.8cm, top = 2.8cm, left = 2.6cm, right = 2.6cm]{geometry}
	
	\usepackage{algorithm}
	\usepackage{algpseudocode}
	\usepackage{algcompatible}
	
	\newcommand{\Dhalf}{\Delta_{J^{1/2}}}
	\newcommand{\R}{\mathbb{R}}
	\renewcommand{\P}{P}
	\renewcommand{\t}{^T}
	\newcommand{\fF}{\mbox{\boldmath  $F$}}
	\newcommand{\fu}{\mbox{\boldmath  $u$}}
	\newcommand{\fv}{\mbox{\boldmath  $v$}}
	\newcommand{\fw}{\mbox{\boldmath  $w$}}
	\newcommand{\fz}{\mbox{\boldmath  $z$}}
	\newcommand{\fy}{\mbox{\boldmath  $y$}}
	\newcommand{\fs}{\mbox{\boldmath  $s$}}
	\newcommand{\fq}{\mbox{\boldmath  $q$}}
	\newcommand{\fx}{\mbox{\boldmath  $x$}}
	\newcommand{\fr}{\mbox{\boldmath  $r$}}
	\newcommand{\ft}{\mbox{\boldmath  $t$}}
	\newcommand{\fe}{\mbox{\boldmath  $e$}}
	
	\newcommand{\hJ}{\mbox{$\hat J_+$}}
	\newcommand{\hJc}{\mbox{$\hat J_c$}}
	\newcommand{\Jp}{\mbox{$J^{1/2}_+$}}
	\newcommand{\Jc}{\mbox{$J^{1/2}_c$}}

	\newtheorem{remark}{Remark}

	\newtheorem{Theorem}{Theorem}[section]
	\newtheorem{Proposition}[Theorem]{Proposition}

	\title{Compact Quasi-Newton Preconditioners for SPD linear systems}
	\author{L. Bergamaschi\footnotemark[2], J. Mar\'{\i}n,\footnotemark[3]
	 \and A. Mart\'{\i}nez\footnotemark[4]}

\begin{document}

	\thispagestyle{plain}
	\maketitle

	\makeatletter
		\smallskip
	\centerline{\bf
		\@date}
		\smallskip
	\makeatother
	\renewcommand{\thefootnote}{\fnsymbol{footnote}}

	\footnotetext[2] {Department of Civil Environmental and Architectural Engineering,
			  University of Padova, Italy, ({\tt luca.bergamaschi@unipd.it}).}
	\footnotetext[3] {Instituto de Matem\'atica Multidisciplinar, Departamento
			  de Matem\'atica Aplicada, 
			  Universidad Polit\'ecnica de Valencia, Spain, ({\tt jmarinma@mat.upv.es}). }
	\footnotetext[4] {Department of Mathematics and Earth Sciences,
			  University of Trieste, Italy, ({\tt amartinez@units.it}).}

	\renewcommand{\thefootnote}{\fnsymbol{footnote}}

	\begin{abstract}
		 In this paper preconditioners for the Conjugate Gradient method are studied to solve the Newton
        system with symmetric positive definite Jacobian. In particular,
        we define a sequence of preconditioners built by means
        of SR1 and BFGS low-rank updates. We develop conditions under which the SR1 update
                maintains the preconditioner SPD.
                Spectral analysis of the SR1 preconditioned Jacobians
        shows an improved eigenvalue distribution as the Newton iteration proceeds.
                A compact matrix formulation of the preconditioner update is developed
                which reduces the cost of its application and is more suitable for parallel implementation.
        Some notes on the implementation of the
        corresponding Inexact Newton method are given and numerical results on
        a number of model problems illustrate the efficiency of the proposed preconditioners.

	\end{abstract}
	\date{\today}
	\section{Introduction}
	The purpose of this work is to develop efficient preconditioners for the linear systems arising in the Inexact Newton method. 
	The Newton method computes the solution of a system of nonlinear equations $\fF(\fx) = 0$, with $\fF:\mathbb{R}^n \rightarrow \mathbb{R}^n$ iteratively by generating a sequence of approximations. Assuming that each component of $\fF$, $f_i$, is differentiable, the Jacobian matrix $J(\fx)$ is defined by
	\[
	J(\fx)_{ij} = \frac{\partial f_i}{\partial x_j}(\fx),
	\]
	and the Newton step is written as
	\begin{equation}\label{equ:newtonstep1}
	\left\{ \begin{array}{ccc}
	J(\fx_k)\fs_k & = & -\fF(\fx_k) \\ 
	\fx_{k+1} & = & \fx_k + \fs_k 
	\end{array} \right. .
	\end{equation} 
	Thus, the method involves the solution of a linear system with the Jacobian matrix. 
	In this work it will be assumed that the Jacobian is a symmetric positive definite (SPD) matrix, as it happens in 
	a number of applications such as, for instance, unconstrained optimization of convex problems,
	discretization of nonlinear PDEs by  e.g.
	the Picard method~\cite{berputijnme98}, or eigensolution of SPD matrices using
	the (simplified) Jacobi-Davidson method~\cite{notay,bmSIMAX13,bm13jam}. When $J$ is large and sparse,  iterative methods are recommended
	to solve the linear system with the Jacobian matrix in (\ref{equ:newtonstep1}).  

	The issue of devising preconditioners for the sequences of linear systems arising in the application
	of an (Inexact) Newton method to nonlinear problems has been addressed e.g. in \cite{MR1176714,MoralesNocedal98}.
	To construct a sequence of preconditioners $\{P_k\}$ for the Preconditioned
	Conjugate Gradient (PCG) solution of the Newton systems we use Quasi-Newton approximations
	of the Jacobian matrix as in \cite{bergamaschi-et-al-06}. In particular we develop and compare
	two approaches based on the symmetric Quasi-Newton formulas: SR1 and BFGS.	
	These low-rank corrections have been used as preconditioners in previous papers. We mention, among the others,
	\cite{NabVui06} where the BFGS preconditioner is named {\em balancing preconditioner}, \cite{freitagIMA}  where
	the authors use a single-vector version of the SR1 update, and the work in \cite{BMnlaa14} where
	both BFGS and SR1 updates are reviewed and used in the framework of iterative eigensolvers.
	In \cite{bbm08sisc} a bounded deterioration property has been proved for the BFGS preconditioners which implies
	that $\|I - P_k J(\fx_k)\|$ can be made as small as desired depending both on the 
	closeness of $\fx_0$ to the exact solution and of the initial preconditioner $P_0$ 
	to the inverse of the initial Jacobian.
	BFGS-based sequences of preconditioners have been successfully used in the acceleration  of inner
	iterative solvers in the framework of eigensolution of large and sparse SPD matrices \cite{bmSIMAX13}.

	Differently from the BFGS sequence the SR1 update formula is not expected to provide SPD preconditioners
	even if the initial one is so. In this work, we will give some some conditions under which it is possible
	to prove symmetric positive definiteness of the SR1 sequence. Moreover, since in the SR1 case the bounded
	deterioration property does not generally hold, we will prove that the spectral
	distribution of the preconditioned matrix at step $k+1$: $P_{k+1} J(\fx_{k+1})$ is not worst than that 
	of the preconditioned matrix at step $k$.

	Application of this kind of  preconditioners requires, in addition to the application of the initial approximation
	of $J(\fx_0)^{-1}$, a number of scalar products that may considerably slow down the PCG iteration,
	even if limited memory variants are considered.
	In view of a parallel implementation, when a high number of processors is employed, large number of scalar products
	usually reveals a bottleneck of the overall efficiency. To address this problem,
we have developed a compact version of the two classes of preconditioners, following the work in \cite{CompBroyNLAA2018} where
analogous formula are developed in a more general framework, but not 
in connection with the idea of preconditioning. These matrix versions of the preconditioner updates
allow for the use of BLAS-2 kernels for their applications and, at the same time, may reduce considerably the number of communications
among processors.

	The remainder of the paper is as follows: in  Section 2 we review the theoretical properties of the
	BFGS sequences of preconditioners and we develop a matrix version  of the preconditioner update.
	Section 3 analyzes the condition under which the SR1  sequence is SPD, shows the theoretical properties
	of the preconditioned matrices and develops a matrix version  of the preconditioner update.
	In Section 4 we present some numerical results which give evidence of the computational gain
	provided by the compact formulas and of the good behavior of the SR1 update to accelerate the PCG method.
	In Section 5 we draw some conclusions.

	\section{BFGS recurrence as a preconditioner}\label{sec:two}
	BFGS-type formulas are used in the context of Quasi-Newton methods to approximate the true Jacobian as the iteration progresses. 
	At a given Newton step, let $B_k$ be an approximation of the Jacobian matrix $J_k$, while $P_k$ denotes its inverse, i.e., $P_k = B_k^{-1}$.
	By denoting $\fF_{k} =  \fF(\fx_{k})$, let us define the vectors $\fy_k = \fF_{k+1} - \fF_k$, $\fs_k = \fx_{k+1}-\fx_k$. The BFGS formula is
	  derived by imposing the following conditions on the inverse Jacobian matrix: 
	 \begin{itemize}
	 \item $P_{k+1}$ must be symmetric and positive definite.
	 \item It must satisfy the secant condition, $P_{k+1}\fy_k = \fs_k$.
	 \item Closeness condition: among all the symmetric matrices satisfying the secant equation, $P_{k+1}$ must be the closest matrix to $P_k$. Using  Frobenius norm:
	 \begin{eqnarray}
	  P_{k+1} = \min_P \| P - P_k \| \label{equ:minimproperty}\\
	  \mbox{subject to} \ \ P = P^T, \ \ P\fy_k = \fs_k. \nonumber
	\end{eqnarray}  
	 \end{itemize}
	The unique solution to (\ref{equ:minimproperty}) is (see \cite{noWr1999}) 
	\begin{equation} \label{equ:invBFGS}
	P_{k+1} = V_k^T P_k V_k + \rho_k \fs_k \fs_k^T \ ,
	\end{equation}
	where 
	\begin{equation}\label{equ:Vk}
	V_k =  \left( I - \rho_k \fy_k \fs_k^T \right) \ , \ \rho_k = 1/\fy_k^T \fs_k .
	\end{equation}
	By applying the Sherman-Morrison-Woodbury formula to (\ref{equ:invBFGS}) an update formula for the  Jacobian matrix approximation $B_k$ is obtained as
	\begin{equation}\label{equ:directBFGS}
	B_{k+1} = B_k - \dfrac{B_k \fs_k \fs_k^T B_k}{\fs_k^T B_k \fs_k} + \dfrac{\fy_k \fy_k^T}{\fy_k^T \fs_k} \ .
	\end{equation}
	
	Moreover, in order to compute the approximate Jacobian efficiently, 
	especially when the number of non-linear iterations is large, limited-memory Quasi-Newton methods where introduced. L-BFGS methods store only a few pairs 
	of vectors $\{ \fs_i, \fy_i \}$ instead of all vectors generated during the iteration process, i.e., early iteration pairs are discarded. Therefore, the computational and memory storage requirements at each Newton iteration are reduced. This may be helpful for solving large-scale nonlinear problems. More precisely, from equations (\ref{equ:invBFGS}) and (\ref{equ:Vk}) we can see that the matrix $P_{k+1}$ is obtained by updating the $P_k$ with the pair 
	$\{ \fs_k, \fy_k \}$. After the new iterate is computed, the oldest vector pair in the set is deleted and replaced by the new pair. Thus, the set of vector 
	pairs includes information from the $k_{\max}$ most recent iterations. 	
	Moreover if $(k\mod k_{\max}) = 0$, a new initial preconditioner is computed as $P_0 \equiv P_k$.
	Starting with an initial matrix $P_0$, formula (\ref{equ:invBFGS}) can be written as
	\begin{eqnarray}\label{equ:H_k}
		P_k & = & \left( V_{k-1}^T \ldots V_{k-k_{\max}}^T \right) P_0 \left( V_{k-k_{\max}} \ldots V_{k-1} \right) \nonumber \\
		& + & \rho_{k-k_{\max}} \left( V_{k-1}^T \ldots V_{k-k_{\max}+1}^T \right) \fs_{k-k_{\max}} \fs_{k-k_{\max}}^T \left( V_{k-k_{\max}+1} \ldots V_{k-1} \right) \nonumber \\
		& + & \rho_{k-k_{\max}+1} \left( V_{k-1}^T \ldots V_{k-k_{\max}+2}^T \right) \fs_{k-k_{\max}+1} \fs_{k-k_{\max}+1}^T \left( V_{k-k_{\max}+2} \ldots V_{k-1} \right) \nonumber \\
	& \vdots & \nonumber \\
	& + & \rho_{k-1}\fs_{k-1} \fs_{k-1}^T \ .
	\end{eqnarray}
	From this expression, the computation of the product $\hat \fr = P_k \fr$ for a given residual vector $\fr$ 
	can be done recursively by performing a sequence of inner products and vector summations involving $\fF_k$ and the pairs $\{ \fs_i, \fy_i \}$, $i=k-k_{\max},\ldots, k-1$, see Algorithm \ref{alg:L-BFGS}. The computational cost is (at most) $2k_{\max}$ daxpys and $2k_{\max}$ dot products. 
	
	\begin{algorithm}[ht]
		\begin{algorithmic}
			\FOR{$i = k-1,  \ldots, k-k_{\max}$} 
			\State  $\alpha_i = \rho_i \fs_i^T \fr$
			\State  $\fr := \fr - \alpha_i \fy_i$
			\ENDFOR
			\State $\hat \fr = P_0 \fr$
			\FOR{$i = k-k_{\max},  \ldots, k-1$} 
			\State  $\beta = \rho_i \fy_i^T \hat \fr$
			\State  $\hat \fr := \hat \fr + (\alpha_i - \beta) \fs_i$
			\ENDFOR
		\end{algorithmic}
		\caption{L-BFGS two-loop recursion to compute $\hat \fr = P_k \fr$}\label{alg:L-BFGS} 
	\end{algorithm}

	The implementation of the Inexact-Newton method~\cite{DemboEisenstatSteihaug82} requires the definition of a stopping criterion for the linear solver (\ref{equ:newtonstep1}) based on the nonlinear residual. Thus, the linear iteration is stopped with the test
	\begin{equation}
	\| J(\fx_k)\fs_k + F(\fx_k)\| \le \eta_k \| F(\fx_k) \| \ .
	\end{equation}
	The sequence $\{\eta_k\}$ will be chosen such that $\eta_k = O(\|\fF(\fx_k) \|)$.
	This will guarantee 
	quadratic convergence of the Inexact Newton method and as a consequence, the following result.
	\begin{Proposition}
	Define $\fe_k = \fx_{\ast} - \fx_k$.
	There exist $\delta > 0$  and $0 < r < 1$ such that if $\|\fe_0\| < \delta $ then
	$\|\fe_{k+1}\| \le r\|\fe_{k}\|$ for every $k$. \\[.4em]
	\label{prop}
	\end{Proposition}
	If the Jacobian matrices are SPD and so is $P_0$, then $P_k$ is also SPD under the condition $\fs_k^T\fy_k > 0$ (see Lemma 4.1.1 in \cite{mybib:bookKelley2}).  Let $\fe_0 = \fx_* - \fx_0$, $E_k = J_k^{-1}- P_{k}$. The next result establishes that $\| I - P_k J(\fx_k) \|$ can be made arbitrarily small by suitable choices of the initial guess $\fx_0$ and the initial preconditioner $P_0$.
	\begin{Theorem}[Theorem 3.6 in \cite{bbm08sisc}]
		For a fixed $\varepsilon_1 > 0$, there are $\delta_0$, $\delta_B$ such that if $\| \fe_0 \| < \delta_0$ and $\|E_0\| < \delta_B$ then
		\[\| I - P_k J_k \|< \varepsilon_1.\]
	\end{Theorem}

\subsection{Compact inverse representation of the BFGS update formula}
Following \cite{byrd-et-al,noWr1999} we write a compact formula for the direct update
\begin{equation}
\label{direct}
B_k = B_0 - 
\begin{bmatrix} B_0 S_k & Y_k \end{bmatrix}
	\begin{bmatrix} S_k^T B_0 S_k & L_k \\ L_k^T & -D_k \end{bmatrix} ^{-1}
\begin{bmatrix}  S_k^T B_0 \\ Y_k^T \end{bmatrix}
\end{equation}
where
\[ S_k = \begin{bmatrix} \fs_0, & \ldots, & \fs_{k-1}\end{bmatrix}, \qquad
Y_k = \begin{bmatrix} \fy_0, & \ldots, & \fy_{k-1}\end{bmatrix} \]
and
\[ (L_k) _{(i,j)} = \begin{cases} \fs_{i-1}^T \fy_{j-1} & \text{if} \ i > j, \\
0  & \text{otherwise} \end{cases}, \qquad D_k = \text{diag}\begin{bmatrix}
\fs_0^T \fy_0, & \ldots, & \fs_{k-1}^T \fy_{k-1}\end{bmatrix} . \]

By applying the Sherman-Woodbury-Morrison formula  to (\ref{direct}) one obtains an explicit formula
for the inverse of $B_k$, i.e.,  $P_k$. We have
\begin{eqnarray*}
	\label{inverse1} 
	P_k &=& P_0 -  \begin{bmatrix} S_k & Z_k \end{bmatrix}
	\begin{bmatrix} R_k^{-T} H_k R_k^{-1} & -R_k^{-T} \\ -R_k^{-1} & 0 \end{bmatrix}
	\begin{bmatrix}  S_k^T  \\ Z_k^T \end{bmatrix} ,
\end{eqnarray*}
where $S_k, Y_k$ are as before and
%
\begin{equation}
	\label{ZRH}
Z_k = P_0 Y_k, \qquad  (R_k) _{i,j} = \begin{cases} \fs_{i-1}^T \fy_{j-1} & \text{if} \ i \le j, \\
0  & \text{otherwise} \end{cases}, \qquad H_k = D_k +  Y_k^T P_0 Y_k .
\end{equation}

\subsection{Computation of $\hat \fr = P_k \fr$}
At step $k$, previously to the PCG solution of the $k$-th Newton system, we have to compute 
$\fz_{k-1} = P_0 \fy_{k-1}$ by solving the linear system $B_0 \fz_{k-1} = \fy_{k-1}$. Then we have to form
recursively
\[ S_k = \begin{bmatrix} S_{k-1}\  \fs_{k-1} \end{bmatrix}, \quad
Z_k = \begin{bmatrix} Z_{k-1}\  \fz_{k-1} \end{bmatrix}, \]
\[R_k = \begin{bmatrix} R_{k-1} &  S_{k-1}^T \fy_{k-1} \\ 0 & \fs_{k-1}^T \fy_{k-1} \end{bmatrix},
\quad
H_k = \begin{bmatrix} H_{k-1} &  Z_{k-1}^T \fy_{k-1} \\ \fy_{k-1}^T Z_{k-1} & (s_{k-1}  + \fz_{k-1})^T \fy_{k-1} \end{bmatrix} .
\]
During the PCG solver, application of $P_k$ to the residual vector takes the
steps indicated in Algorithm \ref{alg:compactPkr-BFGS}.

\begin{algorithm}[ht]
	\begin{algorithmic}
		\State 1. Solve $B_0 \hat{\fr} = \fr$;
		\State 2. $\fw_1 = S_k^T \fr$; 
		\State 3. $\fw_2 = Z_k^T \fr$;
		\State 4. Solve $R_k \fq_2 = \fw_1;$ 
		\State 5. Solve $R_k^T \fq_1 = \fw_2 - H_k \fq_2$;
		\State 5. $\hat{\fr}  = \hat{\fr} - S_k \fq_1 - Z_k \fq_2$.
	\end{algorithmic}
	\caption{Computation of $\hat \fr = P_k \fr$ with compact (L-)BFGS}\label{alg:compactPkr-BFGS} 
\end{algorithm}

\subsection{Limited memory implementation L-BFGS}
If $k > k_{\max} $ then the update of matrices $S_k, Z_k, U_k$ and $H_k$ changes as follows
(in MATLAB notation). First, rewrite the matrices by shifting elements corresponding to the last $k_{\max}-1$ updates: 
\begin{eqnarray*}
	S_{k-1} & \leftarrow & S_{k-1}(:,2:k_{\max}) \\
	Z_{k-1} & \leftarrow & Z_{k-1}(:,2:k_{\max})\\
	R_{k-1} &\leftarrow & R_{k-1}(2:k_{\max}, 2:k_{\max}) \\
	H_{k-1} & \leftarrow & H_{k-1}(2:k_{\max}, 2:k_{\max}).
\end{eqnarray*}
Then, update the matrices with the new entries,
\[ S_k = \begin{bmatrix} S_{k-1}\  \fs_{k-1} \end{bmatrix}, \quad
Z_k = \begin{bmatrix} Z_{k-1}\  \fz_{k-1} \end{bmatrix}, \]
\[R_k = \begin{bmatrix} R_{k-1} &  S_{k-1}^T \fy_{k-1} \\ 0 & \fs_{k-1}^T \fy_{k-1} \end{bmatrix},
\quad
H_k = \begin{bmatrix} H_{k-1} &  Z_{k-1}^T \fy_{k-1} \\ \fy_{k-1}^T Z_{k-1} & (s_{k-1}  + \fz_{k-1})^T \fy_{k-1}  \end{bmatrix} .
\]

\section{Preconditioner based on SR1 recurrence}\label{sec:sr1}

In the L-BFGS updating formula, the matrix $B_{k+1}$ (or $P_{k+1}$) differs from its predecessor $B_{k}$ (or $P_{k}$) by a 
rank$-2$ matrix. There is a simpler rank$-1$ update that maintains symmetry of the matrix and allows it to satisfy the secant equation. But unlike the rank-2 formulas, this symmetric-rank-1 update, also called SR1,  does not guarantee that the updated matrix maintains positive definiteness. 
The direct SR1 update formula is given by (see \cite{noWr1999})
\begin{equation} \label{equ:directSR1}
B_{k+1} = B_{k} + \dfrac{(\fy_k - B_k \fs_k)(\fy_k - B_k \fs_k)^T}{(\fy_k - B_k \fs_k)^T \fs_k} \ .
\end{equation}
By applying the Sherman-Morrison formula, one can obtain the corresponding update formula for the preconditioner
(we will omit subscript $_k$ in vectors $\fs$ and $\fy$ from now on)
	\begin{equation}
	\P_{k+1} =P_k  + \frac{(\fs - P_k \fy) (\fs - P_k \fy)^T}{\fy^T  (\fs - P_k \fy)}
	\label{SR1}
	\end{equation}
	satisfying the secant condition
	\[ P_{k+1} \fy = \fs.\]

	Formula (\ref{SR1}) does not generally provide a sequence of SPD matrices and may not be well defined
	since the scalar $\fy^T  (\fs - P_k \fy)$ can in principle be zero.
	In the sequel of this Section we will discuss these two issues and we will prove
	that $P_{k+1}J(\fx_{k+1})$ has a more favorable eigenvalue distribution than
	 $P_{k}J(\fx_{k})$.

	Let us denote with $\Omega$ an open subset of $\R^{n}$,
	we will make the following {\em standard assumptions} on $\fF$ which we will assume to hold throughout this section.\\[-.2em]
	\begin{figure}[h!]
	{\bf Standard Assumptions:}\\[-.4em]
	\begin{enumerate}
	\item  Equation $\fF(\fx) = 0$ has a solution $\fx_*$.
	\item $J(\fx): \Omega \to \R^{n \times n}$ is Lipschitz continuous
			 with Lipschitz constant $\gamma$.
		 \item $J_*\equiv J(\fx_*)$ is nonsingular (Set $\alpha = \left\|{J_*}^{-1}\right\|$) .
	\end{enumerate}
	\end{figure}

	\noindent
	{\bf Notation.}
	It is now convenient to indicate
	with the subscript $c$ every vector or matrix referring to the
	current iterate and with the subscript $+$ every quantity referring
	to the new iterate $k+1$. Moreover we will use $J_+$ and $J_c$ instead of $J(\fx_{k+1})$ and $J(\fx_k)$
	and $\fF_+, \fF_c$ instead of $\fF(\fx_{k+1})$, $\fF(\fx_k)$.

	Following this notation  Newton's method
	can be stated as
	\begin{eqnarray}
	J_c \fs &=& -\fF_c  \nonumber \\[-0.5em]
	\label{exnewt1} \\[-0.5em]
	\fx_+ &=& \fx_c + \fs \nonumber
	\end{eqnarray}
	We define the error vectors $ \
	 \fe_+ =\fx_* -\fx_+, \qquad
	   \fe_c =\fx_* -\fx_c .$
	\begin{lemma}
	Let $\|\fe_c\| \le \delta_0 < \dfrac{1}{\alpha \gamma} $. Then  $J_c$ is invertible and
		\[\|J_c^{-1} \| \le \dfrac{\alpha}{1 - \alpha \gamma\delta_0} \equiv \alpha_c\]
	\label{Banach}
	\end{lemma}
	\begin{proof}
	Follows from  Theorem 1.2.1 in \cite{kelley95}, also known as Banach lemma.
	\end{proof}

		\noindent
		The following Lemma will provide relations between the vectors $\fs, \fy$ and $\fe_c$.
		\begin{lemma}
			\label{lemma}
		There exists $\delta > 0$ such that if 
	\[ 0 < \|\fx_c - \fx_*\| \le \delta \quad  \]
	then the following relations hold
			\begin{eqnarray}
				\label{se}
				\|\fs\| \  & \le & 2 \|\fe_c\| \\
				\label{ys0}
			\fy &=& J_c \fs + \Delta_1 \fs, \qquad \text{with} \quad \|\Delta_1\| \le \gamma \|\fe_c\| \\
				\label{ynorm}
				\|\fy\| & \ge & \frac{1 - \alpha_c \gamma  \delta}{\alpha_c} \|\fs\| \equiv c_1 \|\fs\|
			\end{eqnarray}
		\end{lemma}

		\begin{proof}
			From $\fs = \fe_c - \fe_+$ we have, by Proposition \ref{prop},
			\[ \|\fs \| \le \|\fe_+\| + \|\fe_c\| \le (1+r) \|\fe_c\| \le 2 \|\fe_c\|.\]
	Now from the standard assumptions, and from the fundamental theorem of calculus we have that 
	\begin{equation}
	\label{ys1}
	\fy =  \fF_+ -  \fF_c = \int_0^1 J(\fx_c+t\fs)\fs dt = J_c \fs + \int_0^1 (J(\fx_c+t\fs) - J_c)\fs dt = 
		J_c\fs + \Delta_1 \fs, 
	\end{equation}
	where $\Delta_1 =  \int_0^1 \left( J(\fx_c+t\fs) - J_c \right)dt\ $.
			From the standard assumptions, (\ref{se}) and Proposition \ref{prop}
			\[ \|\Delta_1\| \le \int_0^1 \|\left( J(\fx_c+t\fs) - J_c \right) \| dt \le
			 \int_0^1 \gamma \|\fx_c+t\fs - \fx_c\| dt  =
			 \gamma \int_0^1 t \|\fs\| dt  \le \gamma \|\fe_c\|.
			\]
			Finally from
		\[ \|J_c^{-1}\| \|\fy\| \ge \|J_c^{-1} \fy\| = \|\fs + J_c^{-1} \Delta_1 \fs\| \ge
	\|\fs\| - \|J_c^{-1} \Delta_1 \fs\| \ge \left(1 - \alpha_c \gamma \|\fe_c\|\right) \|\fs\| \]
			if $\delta < \dfrac{1}{\alpha_c \gamma}$ then we have 
		\[ \|\fy\| \ge \frac{\left(1 - \alpha_c \gamma \delta \right) \|\fs\|} {\alpha_c}. \]
		\end{proof}

		\noindent
		\subsection{Well definition and positive definiteness of $P_{k+1}$}\label{sec:sr1_31}
		We now give some evidence that the sequence $\{P_k\}$ can be made SPD provided that $P_0$ is so.
		Assume that $P_c$ is SPD. Then $P_+$ is SPD   if the denominator in (\ref{SR1}) is positive.  
		Using the results of the previous Lemma we can rewrite the denominator in (\ref{SR1}), using
	 $\fs =  J_c^{-1} \fy - J_c^{-1} \Delta_1 \fs$ from (\ref{ys0}),
	as 
	\[ 
	\fy^T  (\fs - P_c \fy) = \fy^T \fs - \fy^T P_c \fy = \fy^T J_c^{-1} \fy - \fy^T P_c \fy - \fy^T J_c^{-1}\Delta_1 \fs. \]
	Dividing both sides by $\|\fy\|^2 $ we have that 
	\begin{equation}
		\label{qj}
	\frac{\fy^T  (\fs - P_c \fy) }{\|\fy\|^2} 
	= \frac{\fy^T J_c^{-1} \fy }{\fy^T \fy} - \frac{\fy^T P_c \fy }{\fy^T \fy} - \frac{\fy^T J_c^{-1}\Delta_1 \fs}{\fy^T \fy} \ge q_{{J^{-1}_c}}(\fy)- q_{P_c}(\fy) - c_2 \delta  \end{equation}
	where $q_A(\fy)$  denotes the Rayleigh quotient for a given SPD matrix $A$ and $c_2 = 
	\dfrac{ \gamma(1 + \alpha_c \gamma \delta)}{c_1^2} $ since
	\[ \left | \frac{\fy^T J_c^{-1}\Delta_1 \fs}{\fy^T \fy}   \right| \le \frac{(1 + \alpha_c \gamma \delta) (\gamma \delta)\|\fs\|^2}
	{c_1^2 \|\fs\|^2} = \frac{(1 + \alpha_c \gamma \delta) \gamma}{c_1^2}  \delta = c_2\delta. \]
\begin{remark}
	\label{rem}
	The expression on the right of (\ref{qj})  is expected to be positive for the following reason.
	The initial preconditioner $P_0$ is usually computed as the inverse of a Cholesky factorization of $A \equiv J(\fx_0)$ i.e. $P_0 = (L L^T)^{-1}$.
	As known this kind of preconditioner  is able to capture the largest eigenvalues of the coefficient matrix rather
	than the lowest ones. As a consequence the  typical spectrum of the preconditioned matrices is an interval 
	$[\alpha, \beta]$ with $\alpha$ close to zero and $\beta$ slightly larger than 1.  
	The sign of the  eigenvalues of $A^{-1} - (L L^T)^{-1} = L^{-T} \left(L^T A^{-1} L - I\right) L^{-1}$ 
	is related to the sign of those of $L^T A^{-1} L - I$ which are therefore larger than $\frac{1}{\beta} - 1$.
	As an example we report the extremal  eigenvalues of the preconditioned 
	 matrix arising from the FD discretization  of the Poisson equation with size $n = 39204$ preconditioned 
	 with the Cholesky factorizations with either no fill-in  or with  drop tolerances of $\tau = 10^{-3}, 10^{-5}$.
	\begin{center}
	\begin{tabular}{c|ccc}
		& $\alpha$ & $\beta$ \\
		\hline
		IC(0) & $8.504 \times 10^{-4}$ & 1.2057 \\
	$\tau = 1e-3$  & $2.253 \times 10^{-2}$ & 1.1445 \\
	$\tau = 1e-5$  & $0.5097              $ & 1.0998 \\
		\end{tabular}
	\end{center}
%
	In practice, it is rather easy to guarantee that 
	\[d = \min_{\fy \in \R^n} q_{A^{-1}}(\fy) - q_{(LL^T)^{-1}}(\fy)  = \lambda_{\min}\left(
	A^{-1} - (L L^T)^{-1}\right)  \]  is bounded away from zero by first roughly estimating
	$\beta$ and then simply scaling the preconditioner 
	factor $L$ by a scalar to satisfy $\beta < 1$. 
\end{remark}
%
%

	\subsection{Effects of the low-rank update on the eigenvalues of $P_+ J_+$}\label{sec:sr1_32}
	The following Theorem will state that the SR1 correction is able to set approximately to 1 one eigenvalue
	of the preconditioned matrix.
	\begin{theorem}
		If $P_c$ is SPD and $\delta$ is small enough so that the conclusions of Lemma \ref{lemma} hold,
		if moreover the denominator in  (\ref{SR1}) is positive 
		(with $\|P_+\| = \alpha_1$),
	then there is $\delta$ such that if 
	$ 0 < \|\fe_c\| \le \delta $
	then
		\[ P_+ J_+ \fs = \fs + \fw, \qquad \text{with} \quad \|\fw\| \le 6 \alpha_1 \gamma \delta^2. \]
	\label{th1}
	\end{theorem}
	\begin{proof}
	The secant condition
		$P_+ \fy = \fs $ can be rewritten using (\ref{ys1}) as
		\[ P_+ (J_c \fs + \Delta_1 \fs) = \fs\]
		which is equivalent to
		\[ P_+ J_+ \fs = \fs - P_+ \Delta_1 \fs + P_+ (J_+ - J_c) \fs  = \fs + \fw\]
		with 
		\[ \|\fw\| \le \alpha_1 \gamma \delta \|\fs\| + \alpha_1 \gamma \|\fs\|^2 \le \ (\text{using }
		\|\fs\| \le 2\|\fe_c\| \le 2 \delta) \ \le
		6 \alpha_1 \gamma  \delta^2. \]
	\end{proof}

The previous theorem shows that the preconditioned matrix $P_+J_+$ has an approximate eigenvector $\fs$
corresponding to the approximate eigenvalue 1, the goodness of the
approximation depending on how $\fx_c$ is close to the exact solution 
$\fx_*$.  Regarding all other eigenvalues we can observe that  $P_+J_+$ is similar to the SPD matrix $\Jp P_+ \Jp$
so we denote by $\hJ$  this new matrix and by $\hJc$ its analogous at step $k$, namely  $\Jc P_c \Jc$.
The eigenvalue distribution of $\hJ$ is described by Theorem \ref{ldist}. We will handle the simpler
case where $J_c$ and $J_+$ commute and we premise the following 
\begin{lemma}
If $J_c$ and $J_+$ commute and 
		$\delta$ is small enough so that the conclusions of Lemma \ref{lemma} hold, then
	\begin{equation}
\hJ = \hJc + \fz \fz^T + \Delta_2
		\label{hatJ}
	\end{equation}
where
\begin{eqnarray}
	\label{z}
	\fz & = & \frac{\Jc (\fs - P_c \fy)}{\sqrt{\fy^T(\fs - P_c \fy)}}, \qquad \Delta_{J^{1/2}} = 
	\Jp - \Jc\\
	\label{Delta2}
	\Delta_2 & = &  \Dhalf P_+  \Dhalf +   \Dhalf P_+  + P_+ \Dhalf \\
	\label{Delta2norm}
	\|\Delta_2 \| &\le& c_3 \delta, \quad \text{with} \quad 
c_3 =  4 \alpha_1 \gamma (1+ \sqrt{\alpha_c}) \delta 
\end{eqnarray}
\end{lemma}
\begin{proof}
	Multiplying (\ref{SR1}) by $\Jp$ on the left and on the right yields
	\[\hJ =  \Jp P_c \Jp  + \frac{\Jp (\fs - P_c \fy) (\fs - P_c \fy)^T\Jp }{\fy^T  (\fs - P_c \fy)}.\]
	Writing $\Jp = \Jc + (\Jp - \Jc)$ and substituting at the right hand side yield (\ref{hatJ}) with
	definitions (\ref{z}) and (\ref{Delta2}). To obtain (\ref{Delta2norm}) we develop
	the following bounds
\begin{eqnarray*}
	\|J_+ ^{1/2} - J_c ^{1/2}\|  &\le& \|J_+-J_c\| \|\left(J_+^{1/2}  + J_c^{1/2}\right) ^{-1}  \|  \le  
2 \gamma \delta  
 \frac{1}{\lambda_{\min} \left(J_+^{1/2} + J_c^{1/2}\right)}  
 \le   \frac{2\gamma \delta}{\lambda_{\min} ( J_c^{1/2})} = 2 \gamma \delta \sqrt{\alpha_c} .
\end{eqnarray*}
Hence 
	\[\|\Delta_2 \| \le \alpha_1\left(4 \gamma^2 \delta^2 \alpha_c + 4 \gamma \delta \sqrt {\alpha_c}\right) \le
	\ (\text{recalling that} \ \delta\gamma\alpha_c < 1) \ \le 4 \alpha_1 \gamma (1+ \sqrt{\alpha_c}) \delta
	\equiv c_3 \delta.\]
\end{proof}

We are able to state the final theorem regarding eigenvalue distribution for $\hJ$, assuming all eigenvalues of $\hJ$ and
	$\hJc$ are ordered increasingly. Denote furthermore with $\fv_1$ the eigenvector of $\hJc$ corresponding to $\lambda_1(\hJc)$
	\begin{theorem}
		\label{ldist}
		If $J_c$ and $J_+$ commute then
		\begin{eqnarray}
			\label{lmin}
			&\lambda_{1} (\hJ) &\ge  \lambda_{1} (\hJc) + \sqrt{\lambda_2(\hJc) - \lambda_1(\hJc)}\left |\fv_1^T \fz\right |  - c_3 \delta  \\
			\label{lk}
			\lambda_{k} (\hJc) - c_3 \delta  \le & \lambda_{k} (\hJ) & \le \lambda_{k+1} (\hJc) + c_3 \delta,
			\quad k = 2, \ldots, n-1, \\
			\label{lmax}
			&\lambda_{n} (\hJ) &\le  \lambda_{n} (\hJc)  +\|\fz\|^2+ c_3 \delta .
		\end{eqnarray}

	\end{theorem}
\begin{proof}
	Bound (\ref{lmin}) comes from Corollary 2.3 in \cite{IpsNad09}. Bounds (\ref{lk}) and (\ref{lmax})
	are a direct consequence of 
Weyl's theorem \cite[Theorem 8.1.8]{gol91}. \end{proof}

\noindent
Using the previous theorem we can state a bounded deterioration property regarding the condition
number of the preconditioned jacobians:
	\begin{corollary}
		If $J_c$ and $J_+$ commute and
		\begin{equation} \label{deltac3}
		\delta < \frac{\sqrt{\lambda_2(\hJc) - \lambda_1(\hJc)}\left |\fv_1^T \fz\right |}{c_3} \end{equation}
		then
		\[ \kappa(\hJ) \le \left(1 + \frac{\|\fz\|^2}{\lambda_n(\hJc)} + O(\delta)\right) \kappa(\hJc). \]
	\end{corollary}
	\begin{proof}
		If (\ref{deltac3}) holds,  then from (\ref{lmin}) we obtain $\lambda_{1} (\hJ) \ge  \lambda_{1} (\hJc) $ hence
		\[ \kappa(\hJ)  = \frac{\lambda_n(\hJ)}{\lambda_1(\hJ)}  \le 
		                  \frac{\lambda_{n} (\hJc)  +\|\fz\|^2+ c_3 \delta }{\lambda_1(\hJc)}
				  = \frac{\lambda_{n} (\hJc)  +\|\fz\|^2+ c_3 \delta }{\lambda_n(\hJc)} \kappa(\hJc). \]
	\end{proof}

Summarizing the findings of this section, $P_+ J_+$ has an approximate eigenvalue at one thus improving the spectral
properties of $P_c J_c$. Moreover, the condition number of $P_+ J_+$ may be only slightly worse than that of $P_c J_c$,
depending on $\|\fz\|$, which in turn can be controlled by the boundedness of the denominator of $\fz$ 
in (\ref{z}) as per Remark \ref{rem}.
\subsection{Compact representation of the SR1 update}
To develop the 
compact representation  of the inverse update (\ref{SR1})  we recall the definitions in (\ref{ZRH}):
\[
Z_k = P_0 Y_k, \qquad  (R_k) _{i,j} = \begin{cases} \fs_{i-1}^T \fy_{j-1} & \text{if} \ i \le j, \\
0  & \text{otherwise} \end{cases}, \qquad H_k = D_k +  Y_k^T P_0 Y_k .
\]
and set
\[{M}_{k} =  R_{k} + R^T_{k} - H_k, \qquad  {Q}_{k} = S_{k}-Z_{k}. \]
Then  the compact formula for the SR1 update in inverse form is \cite{byrd-et-al}
\begin{equation}
	\label{SR1compact}
	P_k = P_0 +  {Q}_{k} {M}_{k}^{-1}  {Q}^T_{k} .
\end{equation}
Matrix ${M}_{k}$ can be computed using the following recursive characterization
\begin{equation}\label{equ:Mkinverse}
{M}_{k} = \left[ \begin{array}{cc}
	{M}_{k-1} & {\fw}_{k-1} \\ 
	{\fw}_{k-1}^T & {\fq}_{k-1}^T\fy_{k-1}
\end{array} \right] \ ,
\end{equation}
where ${\fw}_{k-1} = {Q}_{k-1}^T \fy_{k-1}$.  Therefore, the inverse update only requires the computation of the vector  ${\fq}_{k-1} = \fs_{k-1} - P_0 \fy_{k-1}$ and the matrix update (\ref{equ:Mkinverse}). 

A simple strategy to prevent break down in the SR1 is just skipping the update if the denominator in (\ref{equ:directSR1}) or (\ref{SR1}) is small. That is, the update is applied if
\begin{equation}
	\label{SR1test}
 | \fs_k^T(\fy_k - B_k \fs_k) | \ge r \| \fs_k\| \| \fy_k - B_k \fs_k\| \ ,
\end{equation}
where $r \in (0,1)$ is a small number, for instance $r = 10^{-4}$. 

\subsection{Computation of $\hat \fr = P_k \fr$}
At step $k$, previously to the PCG solution of the $k$-th Newton system, we have to compute 
$ \fq_{k-1} = \fs_{k-1}-\fz_{k-1}$ by solving the linear system $B_0 \fz_{k-1} = \fy_{k-1}$. Then we have to form
recursively
\[ 
{Q}_k = \left[{Q}_{k-1},  {\fq}_{k-1} \right], \quad
{M}_k = \begin{bmatrix} {M}_{k-1} &  {Q}_{k-1}^T \fy_{k-1} \\ \fy_{k-1}^T {Q}_{k-1} & {\fq}_{k-1}^T \fy_{k-1} \end{bmatrix} \ .
\]
Algorithm \ref{alg:compactPkr-SR1} shows the application of $P_k$ to the residual vector $\fr$.
The storage  memory and number of operations needed are just half of the ones required with the compact BFGS formula.
\begin{algorithm}[ht]
	\begin{algorithmic}
		\State 1. Solve $B_0 \hat{\fr} = \fr$;
		\State 2. $\fw_1 = {Q}_k^T \fr$; 
		\State 3. Solve ${M}_k \fw_2 = \fw_1;$ 
		\State 4. $\hat{\fr}  = \hat{\fr} + {Q}_k \fw_2$; 
	\end{algorithmic}
	\caption{Computation of $\hat \fr = P_k \fr$  with compact (L-)SR1}\label{alg:compactPkr-SR1} 
\end{algorithm}
\subsection{Limited memory implementation L-SR1}
If $k > k_{\max} $ then the update of matrices ${Q}$ and ${M}$ changes as follows
(in MATLAB notation):
\begin{eqnarray*}
	{Q}_{k-1} & \leftarrow & {Q}_{k-1}(:,2:k_{\max}) \\
	{M}_{k-1} & \leftarrow & {M}_{k-1}(2:k_{\max}, 2:k_{\max})
\end{eqnarray*}
\[ 
{Q}_k = \left[{Q}_{k-1},  {\fq}_{k-1} \right], \quad
{M}_k = \begin{bmatrix} {M}_{k-1} &  {Q}_{k-1}^T \fy_{k-1} \\ \fy_{k-1}^T {Q}_{k-1} & {\fq}_{k-1}^T \fy_{k-1} \end{bmatrix} \ .
\]
\section{Numerical experiments}\label{sec:experiments}
In this section we present the results of numerical experiments for solving different nonlinear test problems of large size. As initial preconditioner $P_0$ an incomplete Cholesky factorization (IC) of the initial Jacobian was often used. Therefore, its application requires the solution of two triangular linear systems. Additionally, some results with Jacobi preconditioning, i.e., $B_0 = \diag{(J(\fx_0))}$,  are also presented. The numerical experiments are performed with MATLAB running in a Windows OS PC equipped  with an Intel i5-6600k CPU processor and 16Gb RAM. The CPU times are measured in seconds with the functions \texttt{tic} and \texttt{toc}. 

In Sections 4.1 and 4.2 the IC preconditioner was computed with the \texttt{ichol} function with
drop tolerances $0.1$ and $0.01$. For the solution of the linear systems in (\ref{equ:newtonstep1}) the MATLAB
function \texttt{pcg} was used with stopping criterion the relative
residual norm below  $10^{-6}$ and allowing a maximum of 2000 iterations. 
The nonlinear iterations were stopped whenever $\| F(\fx_k) \| \le 10^{-10} \| F(\fx_0) \|$. 
The initial solution is the vector $\fx_0 = \left[0.1, \ldots, 0.1 \right]^T$.

The results are presented either with graphics or tables. 
In the tables, the total number of linear iterations (totlin) and total CPU time required to solve the problem are reported. 
The number of nonlinear iterations (nlit) is recalled in the caption since it is independent of the selected preconditioner.

The objectives of the experiments were
\begin{itemize} \item to compare standard vs compact formulas
		\item to study the convergence behavior with different values of $k_{\max}$ (size of the update pair set or update window)
		\item to compare the acceleration provided by L-BFGS and L-SR1 compact formulas.
\end{itemize}
All these aspects have been investigated considering a number of nonlinear problems.

\subsection{Discrete Bratu  and modified PHI-2 problems}
The  discrete Bratu problem consists in solving the nonlinear problem
$$A \fu = \lambda D(\fu), \ D = \diag{(\exp(u_1), \ldots, \exp(u_n))}$$
where $A$ is an SPD matrix arising from a 2d discretization of the diffusion equation on a unitary domain, an $\lambda$ is a real parameter.

The second problem considered in this work corresponds to the nonlinear problem
$$A \fu = \lambda D(\fu), \ D = \diag{(u_1^3, \ldots, u_n^3)} \ . $$
This problem is a variant of the PHI-2 equation \cite{GravesM} to maintain the Jacobian SPD. As for the Bratu problem, $A$ is the discretization of the Laplacian operator, $\lambda = -1$.

\begin{figure}[h!]
	\begin{center}
		\includegraphics[width=0.6\linewidth]{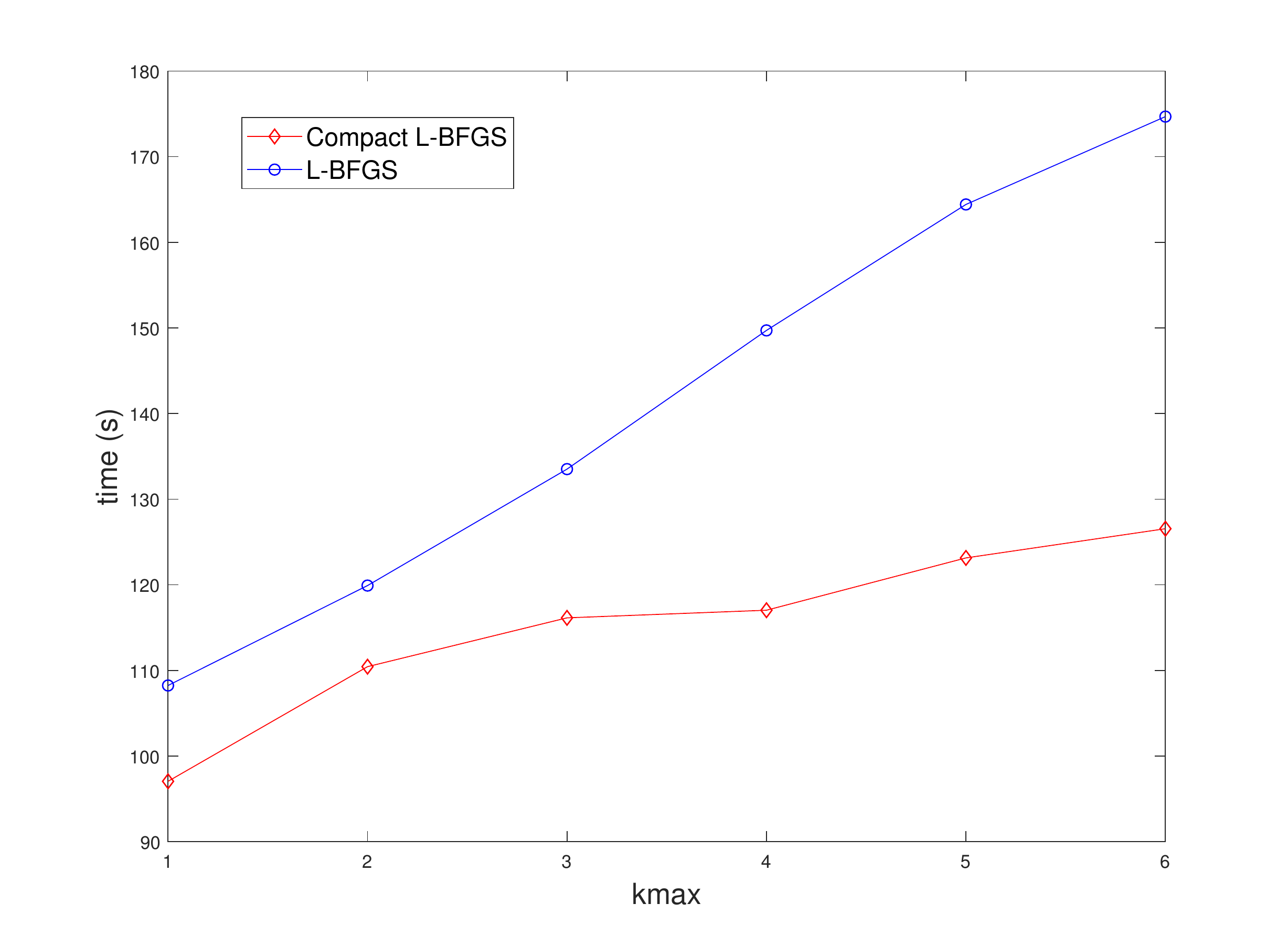}
		\vspace{-4mm}
		\caption{Bratu problem  $n = 747003$, $nnz = 3\,731023 $. Compact vs standard L-BFGS formulas. Initial preconditioner IC(0.01)}\label{figura1}
	\end{center}
\end{figure}

Figure \ref{figura1} shows a comparison between the standard and compact L-BFGS formulas for one of the different discretizations of the Bratu problem tested. It shows clearly that the compact implementation can achieve better computational performance for increasing values of $k_{\max}$, the size of the update window. In our experiments, as it could be expected, it was observed that the difference in performance increases when a large number of linear iterations must be performed, 
situation that is likely to happen when the nonlinear iteration method is approaching the solution (since in these model problems
the ill-conditioning of the jacobians increases toward the end of the Newton process), or when a 
poor initial preconditioner is used. 
Similar results where encountered for the other problems tested, also for the L-SR1 formulas. 
In this case, the gain using the compact formulas is less evident 
since the number of vectors used to update of the preconditioners is half the number of vectors used with L-BFGS.

\begin{figure}[h!]
		\begin{center}
			\begin{minipage}{8.1cm}
	\includegraphics[width=1\linewidth]{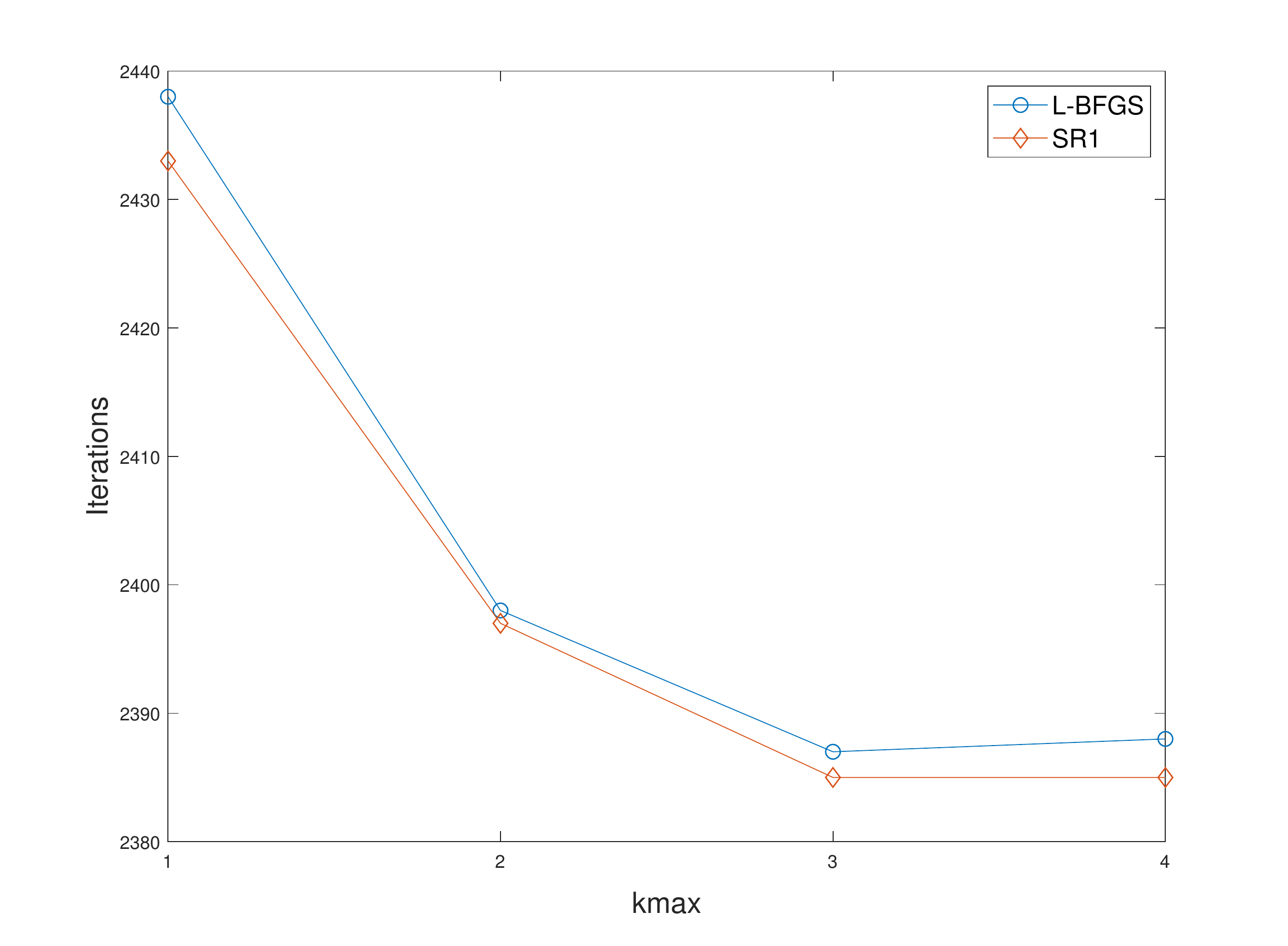}
			\end{minipage}
			\begin{minipage}{8.1cm}
	\includegraphics[width=1\linewidth]{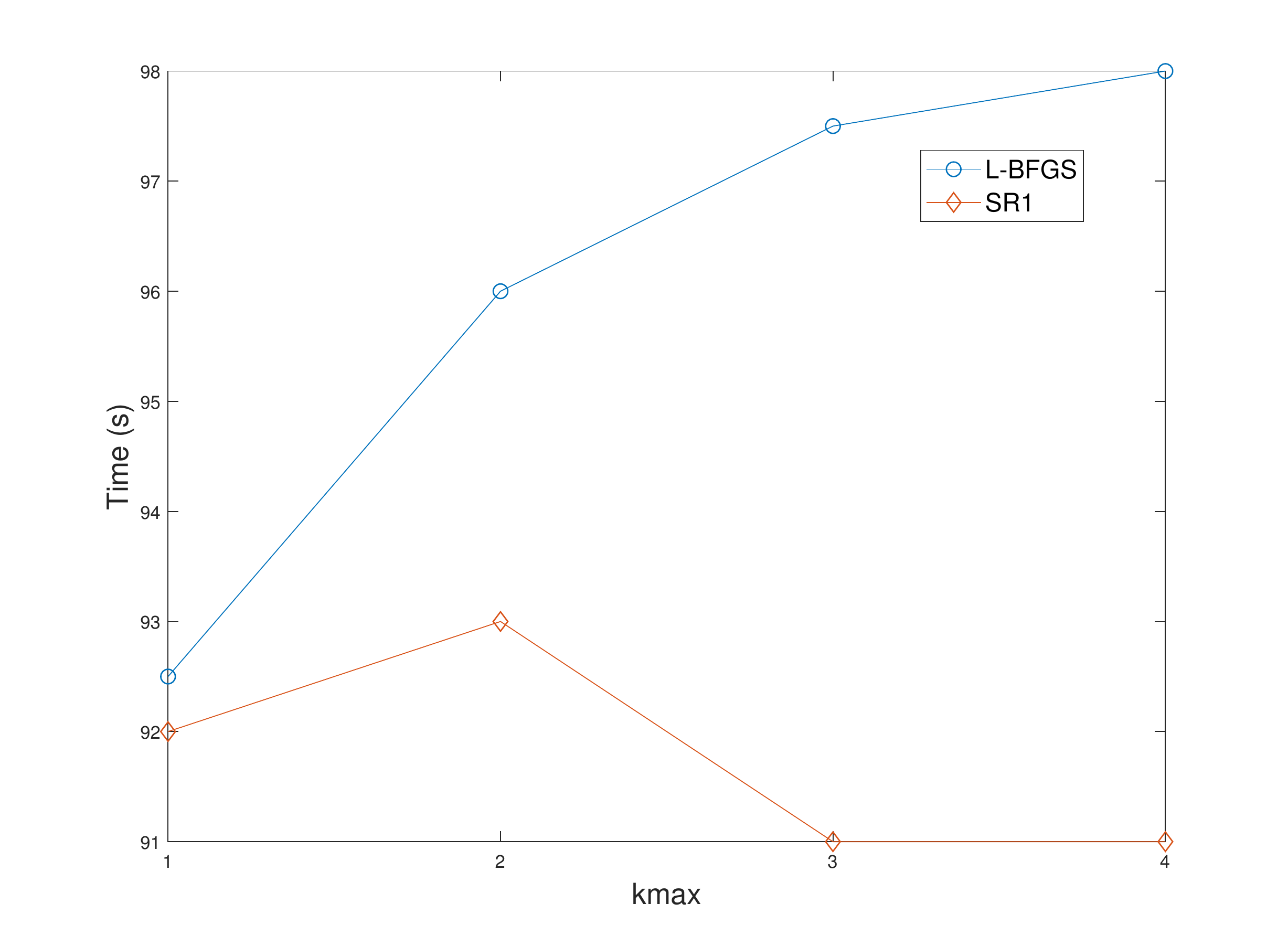}
			\end{minipage}
	\caption{Compact L-BFGS vs L-SR1. Phi-2 problem $n = 747003$, $nnz = 3\,731023 $. Initial preconditioner IC(0.01)}\label{figura2}
		\end{center}
\end{figure}

\noindent
Figure \ref{figura2} shows a comparison between the L-BFGS and L-SR1 formulas in their compact version. It is worth to note that the number of iterations spent by the L-SR1 method is similar, or even smaller in many occasions, than the one required by the L-BFGS. 
This nice behavior gives experimental evidence of the theoretical findings for the SR1 update presented in Section \ref{sec:sr1}. 
As a consequence, the total computational CPU time spent by the Newton method with the L-SR1 was smaller in all our experiments compared to the L-BFGS formula.

\subsection{3d FEM problem}
The third problem considered was a modification of the Bratu problem, where the linear term corresponding to the matrix $A$ arises from a  3D Finite Element discretization of Darcy's law in porous media with heterogeneous hydraulic conductivity coefficient. The size of the problem is $n=268,515$ with a number of nonzero elements  $nnz = 3\, 881337$. For this problem, together IC preconditioning, Jacobi diagonal preconditioning was used to show the effect of starting with a poor preconditioner.

\begin{table}[h!]
	\caption{Results for 3d FEM problem. Number of nonlinear iterations ${\rm nlit} = 15$. Initial preconditioner IC(0.1) -- left, 
		Jacobi -- right.}
	\label{table:3dfem_IC}
	\begin{minipage}{7.3cm}
	\begin{center}
		\begin{tabular}{l|c|c|c}
			preconditioner & $k_{max}$ & totlin & CPU \\
			\hline
			No-update & 0  & $4\,248$ & $66.6$ \\ 
			\hline
			L-BFGS  & 1 & $3\,703$ & $63.5$  \\
			& 2 & $3\,583$ & $67.7$  \\
			& 3 & $3\,519$ & $69.2$  \\
			& 4 & $3\,513$ & $71.1$  \\
			\hline
			L-SR1     & 1 & $3\,673$ & $63.7$  \\
			& 2 & $3\,518$ & $60.9$  \\
			& 3 & $3\,501$ & $62.3$  \\
			& 4 & $3\,493$ & $63.2$  \\
		\end{tabular}
	\end{center}

	\end{minipage}
	\begin{minipage}{7.3cm}
	\begin{center}
		\begin{tabular}{l|c|c|c}
			preconditioner & $k_{max}$ & totlin & CPU \\
			\hline
			No-update & 0  & $9\,124$ & $133.2$ \\ 
			\hline
			L-BFGS  & 1 & $7\,779$ & $123.9$  \\
			& 2 & $7\,501$ & $136.7$  \\
			& 3 & $7\,362$ & $138.8$  \\
			& 4 & $7\,339$ & $139.4$  \\
			\hline
			L-SR1     & 1 & $7\,781$ & $123.6$  \\
			& 2 & $7\,357$ & $119.6$  \\
			& 3 & $7\,192$ & $119.1$  \\
			& 4 & $7\,197$ & $119.2$  \\
		\end{tabular}
	\end{center}
	\end{minipage}
\end{table}

Table \ref{table:3dfem_IC} 
compares the L-BFGS and L-SR1 methods with initial IC and Jacobi preconditioners and  also the no-update
case. It can be observed, analogously to the previous problems, that updating by means of the L-SR1 formula yields 
to the best performance with respect to both total number of linear iterations and  CPU time. 
The time spent to update the preconditioner was negligible compared with the total amount, so it is not indicated. It is worth noting
that for these experiments the best results were obtained with $k_{max}=2,3$ when using L-SR1.
We finally highlight that in all previous experiments, the  SR1 update was always well defined, 
with the denominator in (\ref{SR1}) positive, and the test (\ref{SR1test}) always satisfied, in all Newton iterations.

\subsection{Application to eigenvalue computation}

Given an SPD matrix $A$, to compute its leftmost eigenpair,
the Newton  method in the unit sphere
\cite{simonciniRQI} or Newton-Grassmann method, constructs a sequence of vectors $\{\fu_k\}$ by  solving the linear systems
\begin{eqnarray}
        \label{gras}
	J_k \fs = - \fr_k, \qquad \text{where}\quad   && J_k = (I - \fu_k \fu_k\t) (A - \theta_k I) (I - \fu_k \fu_k\t), \\ \nonumber
	 && \fr_k =  - (A \fu_k - \theta_k \fu_k),  \qquad \text{and} \quad
                                      \theta_k = \dfrac{\fu_k\t A \fu_k}{\fu_k\t \fu_k}
\end{eqnarray}
on a subspace orthogonal to $\fu_k$. Then the next
approximation is set as $\fu_{k+1} = \ft \|\ft\|^{-1}$ where $\ft  = \fu_k + \fs $. Linear system (\ref{gras})
is shown to be better conditioned than the one with $A - \theta_k I$.
The same linear system represents the {\em correction equation} in the well-known Jacobi-Davidson method~\cite{slejpenvdvSIMAX96},
which in its turn can be viewed as an accelerated Inexact Newton method~\cite{tapia}.
When $A$ is SPD and the leftmost eigenpairs are being sought, it has been proved
in \cite{notay} that the Preconditioned Conjugate Gradient (PCG) method can be employed in the solution
of the correction equation. In  this work we follow the PCG implementation of \cite[Algorithm 5.1]{notay}.

We tried our updated preconditioners for the PCG solver within the Newton-Grassmann method 
to compute the leftmost eigenpair
of two large and sparse SPD matrices whose characteristics are reported in Table
\ref{list}. Note that the solution to (\ref{gras}) is not needed to be very accurate.
We then set as the maximum number of inner CG iterations to 50.
The outer Newton process is stopped instead when the following test is satisfied
\[ \|A \fu_k - \theta_k \fu_k\|   = \|\fr_k\|  < \theta_k 10^{-8}.\]
The initial Newton vector  has been computed satisfying $\|\fr_k\| <\theta_k 10^{-2}$.
In both cases the preconditioner is computed once and for all at the beginning of the Newton
iteration as an incomplete Cholesky factorization  with fill-in of {\em matrix $A$}.
\begin{table}[h!]
                \caption{Main characteristics of the matrices used in the tests.}
        \label{list}
        \begin{center}
                \begin{tabular}{llrr|l}
                        \hline
			Matrix & where it comes from & $n$ & $nz$ & drop tolerance \\
                        \hline
			{\sc monte-carlo} & 2D-MFE stochastic PDE   & 77120 & 384320  & $10^{-3}$\\
			{\sc emilia-923} & 3D-FE elasticity problem &  923136  &  41\,005206 & $10^{-5}$ \\
                        \hline
        \end{tabular}
        \end{center}
\end{table}

\subsection*{Matrix \texttt{monte-carlo}}
In Table \ref{monte-carlo} we report the number of outer iterations ({\rm nlit}) the overall number of the linear iterations
(totlin) and the CPU time.
We compared the no-update case and the L-BFGS and L-SR1 updates, both
with $k_{\max} = \{5, 10\}$. Moreover we plot, in Figure \ref{convprof} the nonlinear relative
residual $\|\fr_k\| \theta_k^{-1}$ vs the number of cumulative linear iterations across all the outer Newton iterations.
\begin{table}[h!]
	\caption{Outer/inner iterations an CPU time to evaluate the leftmost eigenpair of matrix \texttt{monte-carlo}.}
	\label{monte-carlo}
	\begin{center}
\begin{tabular}{llll|lll}
 & \multicolumn{3}{c}{$k_{\max} = 10$} &\multicolumn{3}{c}{$k_{\max} = 5$} \\
	\hline
 & nlit & totlin & CPU &nlit & totlin & CPU \\
\hline
L-BFGS  &14  &269   & 4.83 & 15&329&5.66\\
L-SR1   &13  &312  & 4.72  & 14&340&5.07\\
\hline
No Update & 21&  585&  7.50  &&& \\
\end{tabular}
	\end{center}
\end{table}

\begin{figure}[h!]
	\caption{Nonlinear converge profile vs cumulative linear iteration number with different
	preconditioning strategies. Matrix \texttt{monte-carlo}.}
	\label{convprof}
\begin{center}
\includegraphics[width=.72\textwidth]{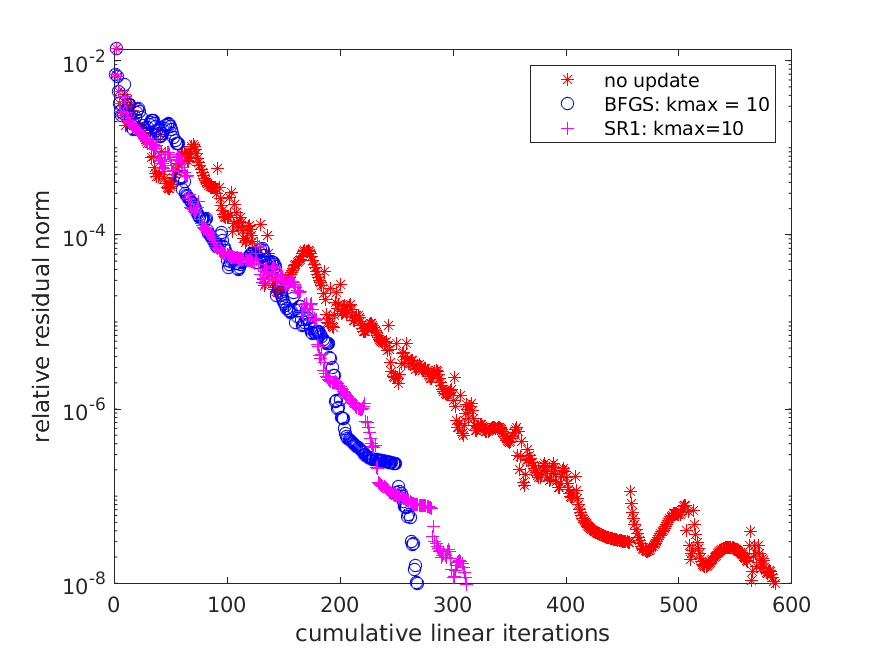}
\end{center}
\end{figure}
From Table \ref{monte-carlo} and Figure \ref{convprof} we can appreciate the remarkable improvement
provided by the compact low-rank updates in terms of inner/outer iterations and CPU time.
Moreover, once again, the L-SR1 update provides comparable results as the L-BFGS correction being (slightly)
the most convenient option in terms of CPU time.

\subsection*{Matrix \texttt{emilia-923}}
We report in Table \ref{emilia} the results for the larger matrix \texttt{emilia-923} where
the L-BFGS acceleration is shown to provide outstanding improvement in both inner/outer iteration
number and CPU time, which is also
accounted for by Figure \ref{convprofEmilia}. The L-SR1 update did not lead to convergence (in fact at 
nonlinear iteration \# 28 the PCG method stopped due to a breakdown). 
\begin{table}[h!]
	\caption{Outer/inner iterations an CPU time to evaluate the leftmost eigenpair of matrix \texttt{emilia-923}.
	For all the update formulas $k_{\max} = 10$ has been set.}
	\label{emilia}
	\begin{center}
\begin{tabular}{llrl}
 & nlit & totlin & CPU \\
\hline
L-BFGS              &    15&  224 &183.22 \\
	L-SR1 (no scaling) &    \dag  &\dag& \dag\\
	L-SR1 (with scaling)&    16  &286& 194.82\\
\hline
	No Update   & 21 &  647&  371.20 \\
	\hline
	\multicolumn{4}{l}{\dag: PCG breakdown found at outer iteration \# 28.}
\end{tabular}
	\end{center}
\end{table}

\noindent
In the light of 
Remark~\ref{rem}, to ensure the positive definiteness of the L-SR1 update, we scaled in this
case the Cholesky triangular matrix by a factor 1.4, after roughly estimating the largest 
eigenvalue of $(LL^T)^{-1} J_0$.  
With this simple modification, the L-SR1 update provides a similar performance as that of  the L-BFGS
update.

\begin{figure}[h!]
	\caption{Nonlinear converge profile vs cumulative linear iteration number with different
	preconditioning strategies. Matrix \texttt{emilia-923}.}
	\label{convprofEmilia}
\begin{center}
\includegraphics[width=.72\textwidth]{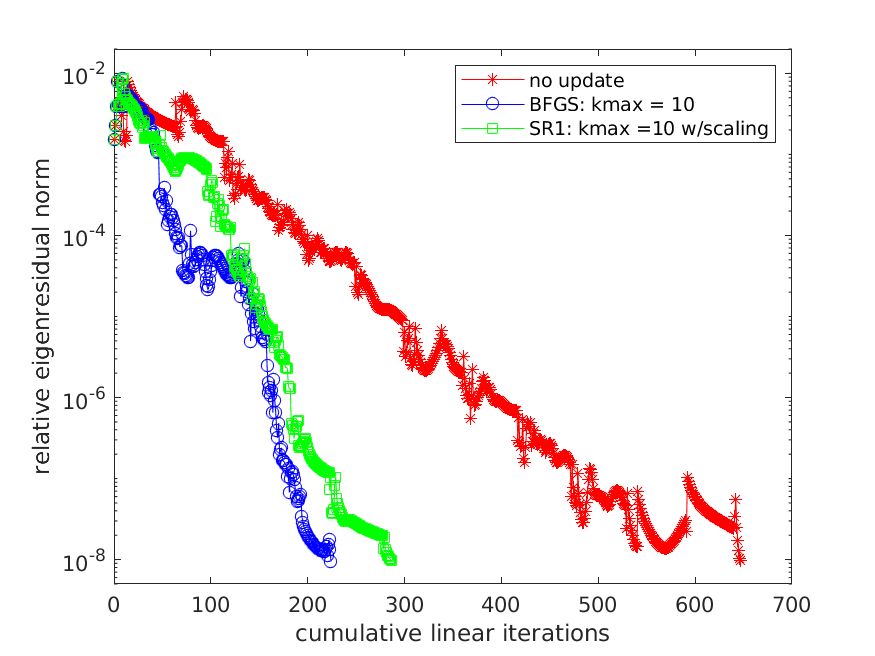}
\end{center}
\end{figure}

As a general comment, regarding Newton method in eigenvalue computation, the effects of the low-rank updates are much
more pronounced as compared to the previous nonlinear problems. Moreover, the relatively high number
of nonlinear iterations and the close to indefinite linear systems to be solved at each Newton step suggest the use
of higher values of the $k_{\max}$ parameter to improve the properties of the low-rank updates.
\section{Conclusions}\label{sec:conclusions}

In this paper a compact version of the rank-two L-BFGS and the rank-one SR1 formulas has
been used to update an initial preconditioner for the solution of nonlinear systems within the Inexact Newton methods. 
One purpose of this study was to show that the compact formulation,
which allows application of the preconditioner in matrix form, can improve the performance obtained in the PCG iteration.
The other objective of the paper was to show that the rank-one update SR1, despite not being guaranteed to
provide SPD preconditioners in all cases, can yet provide SPD sequences 
by simply scaling the initial (Cholesky) preconditioner.
Moreover, it has been shown that the SR1-updated preconditioner is able to shift some eigenvalues towards one, 
and a bounded deterioration property regarding the condition number of the preconditioned jacobians has been stated.

In practice, the construction of the sequence of preconditioners is based on well known limited memory techniques in order 
to keep under control the amount of memory, and also the computational time needed 
to compute and apply the update. From the numerical experiments we can conclude that both the L-BFGS and L-SR1 formulas 
benefit from the use of compact (block) forms. 
More interestingly, the experiments have shown that the L-SR1 may be a good alternative to the L-BFGS formula for 
nonlinear problems with SPD Jacobian.
Work is undergoing to provide a parallel implementation of the compact limited memory Quasi-Newton preconditioners
to improve a given sparse approximate inverse initial preconditioner in the framework e.g. of sequences
of linear systems arising in the eigensolution of very large and sparse matrices in the lines of e.g.~\cite{bm13jam,BMParco17}.

\section*{Acknowledgements}
This work was supported by the Spanish Ministerio de Econom\'{\i}a y Competitividad
under grants MTM2014-58159-P, MTM2017-85669-P and MTM2017-90682-REDT. The first and third authors have
been also partially supported by the INdAM Research group GNCS.
%
%


\begin{thebibliography}{10}
		\begin{normalsize}

\bibitem{bbm08sisc}
{\sc L.~Bergamaschi, R.~Bru, and A.~Mart\'{\i}nez}, {\em Low-rank update of
  preconditioners for the inexact {N}ewton method with {SPD} jacobian},
  Mathematical and Computer Modelling, 54 (2011), pp.~1863--1873.

\bibitem{bergamaschi-et-al-06}
{\sc L.~Bergamaschi, R.~Bru, A.~Mart\'{\i}nez, and M.~Putti}, {\em
  Quasi-{N}ewton preconditioners for the inexact {N}ewton method}, Electron.
  Trans. Numer. Anal., 23 (2006), pp.~76--87.

\bibitem{bm13jam}
{\sc L.~Bergamaschi and A.~Mart\'{\i}nez}, {\em Parallel {RFSAI-BFGS}
  preconditioners for large symmetric eigenproblems}, J. Applied Mathematics,
  2013, Article ID 767042, 10 pages (2013).

\bibitem{bmSIMAX13}
\leavevmode\vrule height 2pt depth -1.6pt width 23pt, {\em Efficiently
  preconditioned inexact {N}ewton methods for large symmetric eigenvalue
  problems}, Optimization Methods \& Software, 30 (2015), pp.~301--322.

\bibitem{BMParco17}
\leavevmode\vrule height 2pt depth -1.6pt width 23pt, {\em Spectral
  acceleration of parallel iterative eigensolvers for large scale scientific
  computing}, Advances in Parallel Computing, 32 (2018), pp.~107--116.

\bibitem{berputijnme98}
{\sc L.~Bergamaschi and M.~Putti}, {\em Mixed finite elements and {N}ewton-type
  linearization for the solution of {R}ichards equation}, Int. J. Numer.
  Methods Engrg., 45 (1999), pp.~1025--1046.

\bibitem{byrd-et-al}
{\sc R.~H. Byrd, J.~Nocedal, and R.~B. Schnabel}, {\em Representations of
  quasi-{N}ewton matrices and their use in limited memory methods},
  Mathematical Programming, 63 (1994), pp.~129--156.

\bibitem{CompBroyNLAA2018}
{\sc O.~DeGuchy, J.~B. Erway, and R.~F. Marcia}, {\em Compact representation of
  the full {B}royden class of quasi-{N}ewton updates}, Numerical Linear Algebra
  with Applications, 25 (2018).
\newblock e2186.

\bibitem{DemboEisenstatSteihaug82}
{\sc R.~S. Dembo, S.~C. Eisenstat, and T.~Steihaug}, {\em {I}nexact {N}ewton
  methods}, SIAM J. Num. Anal., 19 (1982), pp.~400--408.

\bibitem{freitagIMA}
{\sc M.~A. Freitag and A.~Spence}, {\em A tuned preconditioner for inexact
  inverse iteration applied to {H}ermitian eigenvalue problems}, IMA J. Numer.
  Anal., 28 (2008), pp.~522--551.

\bibitem{gol91}
{\sc G.~H. Golub and C.~F. van Loan}, {\em Matrix Computation}, Johns Hopkins
  University Press, Baltimore, 1991.

\bibitem{GravesM}
{\sc P.~Graves-Morris}, {\em Personal communication},  (2010).

\bibitem{IpsNad09}
{\sc I.~Ipsen and B.~Nadler}, {\em Refined perturbation bounds for eigenvalues
  of {H}ermitian and non-{H}ermitian matrices}, SIAM Journal on Matrix Analysis
  and Applications, 31 (2009), pp.~40--53.

\bibitem{kelley95}
{\sc C.~T. Kelley}, {\em Iterative Methods for Linear and Nonlinear Equations},
  SIAM, Philadelphia, 1995.

\bibitem{mybib:bookKelley2}
{\sc C.~T. Kelley}, {\em Iterative Methods for Optimization}, vol.~18 of
  Frontiers in Applied Mathematics, SIAM, Philadelphia, PA, 1999.

\bibitem{BMnlaa14}
{\sc A.~Mart\'{\i}nez}, {\em Tuned preconditioners for the eigensolution of
  large {SPD} matrices arising in engineering problems}, Numer. Lin. Alg.
  Appl., {23} (2016), pp.~{427--443}.

\bibitem{MR1176714}
{\sc J.~M. Mart{\'{\i}}nez}, {\em A theory of secant preconditioners}, Math.
  Comp., 60 (1993), pp.~681--698.

\bibitem{MoralesNocedal98}
{\sc J.~L. Morales and J.~Nocedal}, {\em Automatic preconditioning by limited
  memory quasi-{N}ewton updating}, SIAM J. Optim., 10 (2000), pp.~1079--1096.

\bibitem{NabVui06}
{\sc R.~Nabben and C.~Vuik}, {\em A comparison of deflation and the balancing
  preconditioner}, SIAM J. Sci. Comput., 27 (2006), pp.~1742--1759.

\bibitem{noWr1999}
{\sc J.~Nocedal and S.~J. Wright}, {\em Numerical optimization}, Springer
  Series in Operations Research, Springer-Verlag, New York, 1999.

\bibitem{notay}
{\sc Y.~Notay}, {\em Combination of {J}acobi-{D}avidson and conjugate gradients
  for the partial symmetric eigenproblem}, Numer. Linear Algebra Appl., 9
  (2002), pp.~21--44.

\bibitem{simonciniRQI}
{\sc V.~Simoncini and L.~Eld{\'e}n}, {\em Inexact {R}ayleigh quotient-type
  methods for eigenvalue computations}, BIT, 42 (2002), pp.~159--182.

\bibitem{slejpenvdvSIMAX96}
{\sc G.~L.~G. Sleijpen and H.~A. van~der Vorst}, {\em A {J}acobi-{D}avidson
  method for linear eigenvalue problems}, SIAM J. Matrix Anal. Appl., 17
  (1996), pp.~401--425.

\bibitem{tapia}
{\sc R.~A. Tapia, J.~E. Dennis, and J.~P. Sch\"{a}fermeyer}, {\em Inverse, shifted
  inverse, and {R}ayleigh quotient iteration as {N}ewton's method}, SIAM
  Review, 60 (2018), pp.~3--55.

		\end{normalsize}
\end{thebibliography}
\newpage
\linespread{1.1}

	\end{document}